\documentclass[11pt]{article}
\usepackage[english]{babel}
\usepackage{amsmath,amssymb,amsthm,url} 
\usepackage[mathscr]{eucal} 
\usepackage{dr-sym}  % own symbols 
\newtheorem{theorem}{Theorem}
\newtheorem{lemma}{Lemma}
\newtheorem{corollary}{Corollary}
\title{Some probabilistic results \\ on width measures of graphs}
\author{ Jakub Mare{\v c}ek \\ \small{School of Computer Science, The University of Nottingham} \\ \small{Jubilee Campus, Nottingham NG8~1BB, UK} \\ \small{\url{http://cs.nott.ac.uk/~jxm/}}}
\date{\today}

\begin{document}
\maketitle

\begin{abstract}
Fixed parameter tractable (FPT) algorithms run in time 
$f(p(x))\poly(\abs{x})$, where $f$ is an arbitrary function 
of some parameter $p$ of the input $x$ and $\poly$ is some polynomial function.
Treewidth, branchwidth, cliquewidth, NLC-width, rankwidth, and booleanwidth are parameters often used in 
the design and analysis of such algorithms for problems on graphs.

We show asymptotically almost surely (aas), booleanwidth $\boolw(G)$ is $O(\rw(G) \log \rw(G))$, where $\rw$ is rankwidth.
More importantly, we show aas $\Omega(n)$ lower bounds on the treewidth, branchwidth, cliquewidth, NLC-width, and rankwidth of graphs drawn from a simple random model. 
This raises important questions about the generality of FPT algorithms using the corresponding decompositions.

\end{abstract}

%\category{F.1.3}{Theory of Computation}{Complexity Measures and Classes}
%\category{F.2.2}{Analysis of Algorithms}{Nonnumerical Problems}
%\category{G.2}{Discrete Mathematics}{}

%\terms{packing, heap model}
%\keywords{ACM proceedings, \LaTeX, text tagging} % NOT required for Proceedings

\section{The Introduction}

Fixed parameter tractable (FPT)
algorithms run in time $f(p(x))\poly(\abs{x})$, where $f$ is an arbitrary function of some parameter 
$p$ of the input $x$ and $\poly$ is some polynomial function.
Notice that as long as it is safe to assume that $p$ is $O(1)$, the run time is polynomial in the length of the input, 
even with $f$ exponential or worse.
%There are a number of superclasses of problems approximable in polynomial time (PTAS),
%for which there exist FPT algorithms, one for each parameter $p$.
%The union of these classes is %somewhat confusingly denoted FPT.
For problems on graphs, parameter $p$ is usually a measure of complexity of some tree decomposition of a graph,
which is referred to as the graph's width. 

There has been much progress in the development of FPT graph algorithms recently.  
The attention seems to have shifted from treewidth ($\tw$) \cite{PhDSeese,MR1312164} to newer width measures \cite{Survey}: branchwidth ($\bw$) \cite{MR1110468}, 
cliquewidth ($\cw$) \cite{MR1042649}, NLC-width ($\nlc$) \cite{MR1300250},
rankwidth ($\rw$) \cite{MR2232389}, and booleanwidth ($\boolw$) \cite{PreBooleanwidth}.
It is known \cite{MR2148860,MR2384022,PhDJohansson} graph $G$ has
\begin{align}
%\tw(G) + 1 \ge \rw(G) \ge \boolw(G) \\
\nlc(G) \ge \cw(G) \ge \rw(G) %\ge \boolw(G) 
\label{eq:nerovnost} \\
\tw(G) + 2 \ge \bw(G) + 1 \ge \rw(G) %\ge \boolw(G) 
\label{eq:nerovnost2}
\end{align}
except for some trivial exceptions. There are graphs, whose treewidth is unbounded in cliquewidth \cite{MR2148860}, 
as well as graphs, for which cliquewidth is exponential in either rankwidth or booleanwidth \cite{PreBooleanwidth}.
Intriguingly, booleanwidth can be both more or exponentially less than rankwidth \cite{PreBooleanwidth}.
There are $O(n)$ FPT algorithms for obtaining treewidth and branchwidth decompositions \cite{MR1616221,MR1417901}.
For some (presently unknown) $f$, there exists an $f(k) O(n^3)$ algorithm for obtaining rankwidth-$k$ decompositions of a graph on $n$ vertices 
or certifying their non-existence \cite{MR2421076}, which also gives the best known approximation of cliquewidth and booleanwidth.  
Hence, the attraction.

To some extent, however, our understanding of fixed parameter tractability is limited by the very assumption that the parameter is constant.
In this paper, we attempt to use probabilistic methods to study the dependence of width measures on the number of vertices of a graph.
We show asymptotically almost surely, there are $\Omega(n)$ lower bounds on the treewidth, branchwidth, cliquewidth, NLC-width, and 
rankwidth of graphs drawn from a simple random model. 

\section{The Definitions}

We mention only the definitions of rankwidth and booleanwidth we use in the proofs.
For standard definitions of rankwidth and booleanwidth \cite{MR2232389,PreHJoin,PreBooleanwidth}, 
as well as for any other definitions, please follow the references.

We take a more general view of what is a width measure, 
suggested by Robertson and Seymour \cite{MR1110468} and quoted in verbatim from Bui-Xuan et al. \cite{PreBooleanwidth}:
Let $f$ be a cut function of a graph $G$, and $(T, \delta)$ a decomposition tree of $G$. 
For every edge $uv$ in $T$ , $\{X_u,X_v\}$ denotes the 2-partitions of $V$ induced by the 
leaf sets of the two subtrees we get by removing $uv$ from $T$. The $f$-width of $(T, \delta)$
is the maximum value of $f(X_u)$, taken over every edge $uv$ of $T$ . An optimal $f$-decomposition 
of $G$ is a decomposition tree of G having minimum $f$-width. The $f$-width of $G$ is the $f$-width 
of an optimal $f$-decomposition of G.

In this framework, it is easy to define rankwidth and booleanwidth.
In rankwidth, the function $f$ is the cut-rank function:
\begin{align}
f_{\rw} = \log_2 \left| {\left\{ {Y \subseteq B:\exists X \subseteq A,Y = \bigtriangleunion\limits_{x \in X} {N(x)} } \right\}} \right|,
\label{eq:cutrank}
\end{align}
where neighborhood $N(x)$ are vertices adjacent to $x$ and $\bigtriangleunion$ denotes the set difference.
This is the base-2 logarithm of the size of the row space over $GF(2)$-sums, which 
is the number of pairwise different vectors that are spanned by the rows of the $\abs{A} \times \abs{V \setminus A}$ 
submatrix of the adjacency matrix of $G$ over $GF(2)$, where $0 + 0 = 1 + 1 = 0$ and $0 + 1 = 1 + 0 = 1$.
The corresponding discontiguous definition of taking a submatrix will be used throughout the paper.
Boolean-width can then be defined similarly with
\begin{align}
f_{\boolw} = \log_2 \left| {\left\{ {Y \subseteq B:\exists X \subseteq A,Y = \cupunion\limits_{x \in X} {N(x)} } \right\}} \right|.
\label{eq:boolrank}
\end{align}
Informally, we take the logarithm of the number of distinct unions %(or boolean sums) 
of the neighbourhoods of vertices.
This, without much surprise, is the base-2 logarithm of the size of the row space of a binary matrix with boolean-sums ($1 + 1 = 1$).

Let us now approach rankwidth via the rank of certain submatrices of random matrices over $GF(2)$.

%\section{A Preview of the Technique}
\section{A Lemma}
 
Let us intially use a simple model $M(m,n)$ of random $m \times n, m \le n$ matrices over $GF(2)$, 
where % $M(m,n)$ of $n \times n$ random matrices over $GF(2)$ for $n \in 3, 6, 9, \ldots$, where
each element of a matrix is chosen independently to be $0$ with probability $\frac{1}{2}$
 and $1$ with probability $\frac{1}{2}$.
More rigorously, this is a family of probability spaces over matrices over $GF(2)$. 

First, we state a theorem derived from Bl{\"o}mer, Karp, and Welzl \cite{MR1608234},
which may remind us of Shannon's switching game \cite{MR0073317},
and use it to derive a simple lemma.

\begin{theorem}[Bl{\"o}mer et al. \cite{MR1608234}]
\label{lem:bloemer}
The probability that an $m \times n, m \le n$ matrix $M_{m,n}$ drawn randomly from $M(m,n)$ has rank 
less than $m - d$ is $2^{-\Omega((n - m + d)(m + d))}$.
\end{theorem}

\begin{lemma}
\label{lem:pomocne}
Asymptotically almost surely, the minimum rank of $\frac{n}{3} \times \frac{2n}{3}$ 
submatrices of $M_{n,n}$ drawn randomly from $M(n,n)$ is bounded from below by $\Omega(n)$.
\end{lemma}

%the probability that any random $\frac{n}{2} \times \frac{n}{2}$ matrix over $GF(2)$ will have rank $r \ge \frac{n}{4}$, is $1 - 2^{-\Omega(n^2)}$.
%& = \prod_{N \in U_M} P(\rank(N) \ge \frac{n}{4}) \notag \\

\begin{proof}
Let us denote the minimum rank among $\frac{n}{3} \times \frac{2n}{3}$ submatrices $S_{\frac{n}{3},\frac{2n}{3}}$ in an $n \times n$ matrix 
$M_{n,n}$ over GF(2) drawn from $M(n,n)$ by $\mu$ and let us study the probability of $\mu$ being greater than an arbitrary 
$\frac{n}{6}$. 
Using Boole's inequality (\ref{step:two}), $O(3^{3n})$ bound \cite{MR1543694,MR1876263} on the number of submatrices 
of interest in an $n \times n$ matrix given by binomial coefficient $\binom{n}{\frac{n}{3}}$ (\ref{step:three}),
and Theorem~\ref{lem:bloemer} (\ref{step:four}):
\begin{align}
%\prob(\mu \ge \frac{n}{6}) & = \prob( \rank(N) \ge \frac{n}{6} \quad \forall N \in U_M) \notag \\ 
%& \approx (1 - \prob(\rank(N) < {\frac{n}{6}}))^{\binom{n}{{\frac{n}{3}}}} \notag \\
%& \approx (1 - 2^{-n^2})^{3^{3n}}
\prob(\mu \le \frac{n}{6}) & = \prob( \rank(N) \le \frac{n}{6} \quad \forall S_{\frac{n}{3},\frac{2n}{3}} \in M_{n,n} \label{step:one} \\ 
& \le \sum_{S_{\frac{n}{3},\frac{2n}{3}} \in M_{n,n}} \prob(\rank(S_{\frac{n}{3},\frac{2n}{3}}) \le {\frac{n}{6}}) \label{step:two} \\
& \le {\binom{n}{{\frac{n}{3}}}} \; \prob(\rank(S_{\frac{n}{3},\frac{2n}{3}}) \le {\frac{n}{6}}) \label{step:three} \\
%& \approx (1 - \prob(\rank(N) < {\frac{n}{6}}))^{\binom{n}{{\frac{n}{3}}}} \label{step:two} \\
& \approx O({3^{3n}}{2^{-n^2}}) \label{step:four}
\end{align}
%as there are only $\binom{n}{\frac{n}{3}} \approx O(3^{3n})$ submatrices of interest in an $n \times n$ matrix \cite{MR1543694,MR1876263}.
Clearly, 
\begin{align}
\lim_{n \to \infty}{{3^{3n}2^{-n^2}}} = 0. % \Omega ?!
\end{align}
%and it would be sufficient to consider $r \ge \frac{n}{3} -\Omega(\sqrt n)$.
\end{proof}

\section{The Main Result}

Now, we can state the main result using a simple model $G(n,\half)$ of random graphs of $n$ vertices, 
where each edge appears independently with probability $\frac{1}{2}$.

\begin{theorem}
Asymptotically almost surely, the rankwidth of a graph drawn randomly from $G(n,\half)$ is bounded from below by $\Omega(n)$.
\end{theorem}

\begin{proof}
The minimax theorem of Robertson and Seymour \cite{MR1110468} linking branchwidth and tangles, 
  translated to rankwidth by Oum \cite{MR2384022},
  implies there exists an edge in any decomposition tree, which corresponds to a partition $(V_1, V_2)$ 
  of $n$ vertices of $G$, such that $\frac{n}{2} \ge \abs{V_1} \ge \frac{n}{3}$ and $\frac{n}{2} \le \abs{V_2} \le \frac{2n}{3}$. 
But then the rank of the $\frac{n}{3} \times \frac{2n}{3}$ submatrix of the minimum rank in the adjacency matrix of graph $G$, 
  given by Lemma~\ref{lem:pomocne}, is a lower bound on the value of the rankwidth of $G$.
Notice we need not consider skew-symmetric matrices in Lemma~\ref{lem:pomocne}, as the sub-matrix is not (necessarily) symmetric. 
\end{proof}

Given the trivial upper bound of $n$ on rankwidth of a graph on $n$ vertices, 
it is easy to see this lower bound is tight:

\begin{corollary}
Asymptotically almost surely, the rankwidth of a graph drawn randomly from $G(n,\half)$ is $\Theta(n)$.
\end{corollary}

% Mention which bounds are tight (!)

Finally, using the inequalities between the values of the parameters (\ref{eq:nerovnost}--\ref{eq:nerovnost2}), we can state the following: 

\begin{corollary}
Asymptotically almost surely, the treewidth of a graph drawn randomly from $G(n,\half)$ is $\Theta(n)$.
\end{corollary}

\begin{corollary}
\label{cor:bw}
Asymptotically almost surely, the branchwidth of a graph drawn randomly from $G(n,\half)$ is $\Theta(n)$.
\end{corollary}

\begin{corollary}
Asymptotically almost surely, the cliquewidth of a graph drawn randomly from $G(n,\half)$ is $\Theta(n)$.
\end{corollary}

\begin{corollary}
Asymptotically almost surely, the NLC-width of a graph drawn randomly from $G(n,\half)$ is $\Theta(n)$.
\end{corollary}

It should be noted that Corollary~\ref{cor:bw} seems to be the first probabilistic result on
branchwidth of random graphs. 

\section{Yet Another Bound}

Independently, but still using probabilistic arguments, we can also show:

\begin{theorem}
\label{thm:relationship}
Asymptotically almost surely, booleanwidth $\boolw(G)$ of a graph $G$ is $O(\rw(G) \log \rw(G))$, where $\rw(G)$ is the rankwidth of $G$.
\end{theorem}
\begin{proof}
Bui-Xuan et al. \cite{PreHJoin,PreBooleanwidth} have shown $2^{\boolw(G(A))}$ is bounded from above by the number of subspaces 
$GF(2)$-spanned by the rows (resp. columns) of $A$,
where $G(A)$ is the graph given by the adjacency matrix $A$.
Goldman and Rota \cite{MR0252232} have shown the the number of subspaces of a vector space
corresponds to the number of partitions of a set.
But when we look at the number $c$ of partitions of a set of size $n$,
with probability $1 - o(e^{-n})$ \cite{MR671583}:
\begin{align}
%\log(\E[c]) & = n(\log n - (2 \log {1 \over q} \log n)^{1 \over 2} - {1 \over 2} \log \log n + O(1)) \notag \\
\log c \le n(\log n - \log \log (n - 1) + O(1)).
\end{align}
\end{proof}

%Perhaps more interestingly:
%\begin{theorem}
%\label{thm:relationship}
%Asymptotically almost surely, treewidth and branchwidth %cliquewidth, NLC-width, and rankwidth 
%of graphs drawn randomly from $G(n, p)$, where $p = O(\log n)$, is $O(\log n)$.
%\end{theorem}
%\begin{proof}
%This result follows from the results of Spencer and Riordan on the diameter of sparse graphs:
%***.
%Trivially, diameter of $G$ is an upper bound on branchwidth of $G$. 
%\end{proof}
%This is important in connection with the algorithm of Cunningham and Geelen \cite{ACM1419515} for integer
%programming with constraint matrix corresponding to a hypergraph of bounded branchwidth:
%\begin{corollary}
%Asymptotically almost surely, the run-time of the algorithm of Cunningham and Geelen
%on programs with constraint matrices drawn randomly from $G(n, p)$, where $p = O(\log n)$, 
%is $O(mn^2 + m^2n)$.
%\end{corollary}

\section{The Conclusions}

In this paper, we have used probabilistic methods to study modern width measures of random graphs. 
%\Omega(\binom{n^2}{m})) ~ 
%\cite{MR1138421}
We are aware of only a few results in this direction.
Prior to the unofficial publication of this draft, 
Bodlaender and Kloks \cite{MR1312164} studied treewidth of random graphs and
Johansson \cite{MR1676494} studied NLC-width and cliquewidth of random graphs.
Independently, Gao \cite{Gao2003} studied treewidth of random NK landscapes.
Since the unofficial publication of this draft,
Lee and Lee \cite{PreLeeLee} have provided very elegant proofs of the our results 
and Telle \cite{PreBooleanwidthExpectation} has established the polylogarithmic booleanwidth of random graphs. 
Our results also complement the theorem of Boliac and Lozin \cite{MR2079015}, which implies %\cite{Survey} 
 that for each $k > 1$, the number of graphs having $n$ vertices and clique-width at most $k$ is only $2^{\Theta(n \log n)}$. 

The results suggest the limits of generality of algorithms designed and analysed using five well-known width measures of graphs,
  although there clearly are exponentially large classes of graphs, for which they are very appropriate.
If, however, the runtime is $f(k)\poly(\abs{x})$, where $f$ is exponential or worse 
and there is a $\Omega(n)$ lower bounded to go with $k$, we have not gained much
by making the analysis more detailed.

An important goal for further research is the characterisation of graphs with the
   expected value of some width measures logarithmic in the
   number of vertices, so as to provide some guidance, 
   where can one apply graph decompositions and fixed parameter tractable algorithms successfully.
%It would also be most interesting to know the expected values of various width measures
%In an extended version of this paper, we indend to study the expected values of various width measures
% on other random models of matrices and graphs. %The natural scale-free graphs are the obvious choice. 
Could it be that sparse constraint matrices of large classes of integer programs have branchwidth and rankwidth bounded by $O(\log n)$,
 and hence \cite{ACM1419515} are solvable in polynomial time, for instance?
In random models parametrised with density, it seems interesting to study 
  the behaviour of the expected value of width measures of ``hard'' instances. 
%  which for many problems often seem to be restricted to some narrow intervals of density.
Could there be a relationship with high width measures?
%It also remains open, what is the expected value of booleanwidth \cite{PreBooleanwidth} on graphs in
%the simple model $G_n$.

%By ``phase transition'' we mean, rather unusually, the sharp rise of run time of exact solvers for optimisation problems 
%per instance with density from a certain narrow interval.
%The expected values of width measures of graphs corresponding to hard instances  
%would be particularly informative of the usefulness of present and future width measures.

%Could it be that sparse constraint matrices of large classes of integer programs have branchwidth and rankwidth bounded by $O(\log n)$,
% and hence \cite{ACM1419515} are solvable in time $O(mn^2 + m^2n)$, where $m$ is the number of constraints, for instance?
%In random models parametrised with density, it seems interesting to study 
%  the behaviour of the expected value of width measures of ``hard instances''.
%  which for many problems often seem to be restricted to some narrow intervals of density.
%Could there be a connection?

\paragraph{Acknowledgments}
The author is indebted to Noga Alon \cite{MR2437651,NogaPersonal}, 
  whose beautiful proof inspired Lemma~\ref{lem:pomocne} in this paper,
as well as to Chris Wuthrich, Bjarki Holm, and Sang-il Oum, 
  who have kindly provided comments on early drafts.   
%The proof of Theorem~\ref{thm:relationship} is adapted from Colin McDiarmid \cite{MR522966},
% who has used this to model colourings of random graphs $G_n$, 
% where each edge occurs independently with probability $p$.
Thanks are also due to Petr Hlin{\v e}n{\' y} for his unwavering patience
and unrelenting intellectual stimulation. 

%\newpage
\bibliographystyle{abbrv}
\bibliography{width-prob}

\end{document}